\documentclass[12pt,leqno]{amsart}
\usepackage{amsmath,amsfonts,amssymb,amsthm}
\usepackage{graphicx}
\graphicspath{ }
\usepackage{csquotes}
\usepackage{wrapfig}
\usepackage{accents}
\usepackage{caption}
\usepackage{subcaption}
\usepackage{calligra}
\usepackage[colorlinks]{hyperref}
\hypersetup{citecolor=black}

\numberwithin{figure}{section}

\theoremstyle{plain}
\newtheorem{thm}{Theorem}[section]
\newtheorem{lem}[thm]{Lemma}
\newtheorem{prop}[thm]{Proposition}
\newtheorem{cor}{Corollary}[thm]
\theoremstyle{definition}
\newtheorem{defn}{Definition}[section]

\newtheorem{exmp}{Example}[section]

\theoremstyle{remark}

\newtheorem*{note}{Note}
\usepackage{mathtools}
\title{Generalized almost statistical convergence}
\author[A. A. Shaikh]{Absos Ali Shaikh$^{1}$}
\address{Department of Mathematics\\ The University of Burdwan\\ Burdwan--713101\\ West Bengal\\ India.}
\email{$^1$aask2003@yahoo.co.in, aashaikh@math.buruniv.ac.in}
\author[B. R. Datta]{Biswa Ranjan Datta$^2$}
\address{Department of Mathematics\\ The University of Burdwan\\ Burdwan--713101\\ West Bengal\\ India.}
\email{$^2$math.biswa@gmail.com}
\begin{document} 
\maketitle
\begin{abstract}
	The objective of this paper is to introduce the notion of \textit{generalized almost statistical} (briefly, GAS) convergence of bounded real sequences, which generalizes the notion of almost convergence as well as statistical convergence of bounded real sequences. As a special kind of Banach limit functional, we also introduce the concept of Banach statistical limit functional and the notion of GAS convergence mainly depends on the existence of Banach statistical limit functional. We prove the existence of Banach statistical limit functional. Then we have shown the existence of a GAS convergent sequence, which is neither statistical convergent nor almost convergent. Also, some topological properties of the space of all GAS convergent sequences are investigated.
\end{abstract}

\section{Introduction}
The theory of convergence arises from the study of the situation when the indices of a sequence become larger and larger.  A sequence $(\xi_n)$ of real numbers is said to be convergent to a real number $l$ if for any $\epsilon>0$ there exists an integer $n_0\in\mathbb N$ such that $|\xi_k-l|<\epsilon$ for all $k>n_0$. But this usual notion of convergence can not always capture the detailed characteristics of a non-convergent sequence. Because there are various non-convergent sequences which are ``nearly to be" convergent. There are several generalizations of usual convergence, viz. almost convergence~\cite{banach,lorentz}, statistical convergence~\cite{fast,steinhaus} etc. Nevertheless, it is always better to have larger set of convergent sequences in more generalized sense under investigation.

The existence of Banach limit functionals was proven by Banach~\cite{banach} in 1932. Using Banach limits, in 1948, Lorentz~\cite{lorentz} introduced the notion of almost convergence, which is a generalization of usual convergence of real sequences. Moricz and Rhoades~\cite{moricz} extended the idea of almost convergence for double sequences and later it was studied in \cite{mahiuddine}. Again in 1951 Fast~\cite{fast} and Steinhaus~\cite{steinhaus}  introduced independently the notion of statistical convergence by a rigorous use of natural density of subsets of $\mathbb N$, which is another generalization of usual convergence. Salat~\cite{salat}, Fridy~\cite{fridy,fridy1}, Miller~\cite{miller} and many others~\cite{maddox,erdos,connor,fridyorhan2} studied the concept of statistical convergence. In last few decades, several types of generalizations of usual convergence have been introduced and lots of research papers can be found in literature. Mursaleen~\cite{mursaleenlambda} introduced the idea of $\lambda$-statistical convergence in 2000. If the sequence $\lambda$ is choosen, particularly, by $\lambda=(1,2,3,4,\dots)$, then $\lambda$-statistical convergence coincides with the statistical convergence.  In 2001 Kostyrko et al. \cite{kostyyrkosalat2001} introduced the notion of ideal convergence of real sequences. If we consider the ideal of all subsets of $\mathbb N$ having natural density zero, then the ideal convergence coincides with statistical convergence. Later Lahiri and Das~\cite{lahiridas} extended the concept of ideal convergence for nets in topological spaces. In 2003, Mursaleen and Edely~\cite{mursaleen} extended the idea of statistical convergence for real double sequences. Further some generalizations of usual convergence were introduced and studied in \cite{yurdakadim,koga}.

In this paper, the existence of a certain type of Banach limit functionals is shown which are designated as \textit{Banach statistical limit} functionals. With the help of Banach statistical limit functionals we have introduced \textit{generalized almost statistical} (briefly, GAS) convergence of bounded real sequences.  Then we have shown that GAS convergence is a generalization of almost convergence as well as statistical convergence. Hence a natural question arises: does there exist any GAS convergent sequence which is neither almost convergent nor statistical convergent? The answer to the above question is affirmative and the existence of such sequence is shown in Example~\ref{ex:onlygalst}. At the end, we also investigated some topological properties of the space of all GAS convergent sequences. The outline of this paper is as follows: Section 2 deals with some rudimentary facts about Banach limit functionals, almost convergence, natural density, statistical convergence etc. as preliminaries. Section 3 is devoted to the investigation of main results of the paper (see Theorem~\ref{th1}) and the conclusion of the hole paper is drawn in the last section.

\section{Preliminaries}
The limit functional $f$ on the space $c$ of all convergent real sequences defined by $f(x)=\lim\limits_{n\to \infty}x_n,\ x\in c$, can be extended to the space $l_\infty$ of all bounded real sequences by Hahn-Banach Extension theorem, where $l_\infty$ is the normed linear space with sup-norm defined by $||x||_\infty=\sup\limits_{n\in \mathbb N}|x_n|$ for all $x\in l_\infty$. Banach~\cite{banach} showed the existence of certain extensions which are called Banach limits, defined as follows:
\begin{defn}
	A functional $B:l_\infty \to \mathbb R$ is called Banach limit if it satisfies the following :
	
	(i) $||B||=1$,
	
	(ii) $B|_c=f$, where $f$ is the limit functional on $c$,
	
	(iii) If $x\in l_\infty$ with $x_n\geq0$ $\forall n\in\mathbb N$, then $B(x)\geq0$,
	
	(iv) If $x\in l_\infty$, then $B(x)=B(Sx)$,	where $S$ is the shift operator defined by $S((x_n)_{n=1}^\infty)=S((x_n)_{n=2}^\infty)$.
\end{defn}

	The concept of almost convergence, which is a generalization of usual convergence, was introduced by Lorentz~\cite{lorentz} in 1948 by using Banach limit functionals.

\begin{defn}
	For some $x\in l_\infty$, if $B(x)$ is unique (i.e. invariant) for all Banach limit functionals $B$, then the sequence $x$ is called almost convergent to $B(x)$.
\end{defn}

Let $\mathcal A\subset l_\infty$ be the set of all almost convergent real sequences. Clearly $c\subsetneq \mathcal A$.

\begin{defn}
	Let $P\subset\mathbb N$. If the limit $$\delta(P)=\lim\limits_{n\to \infty}\frac{|P\cap\{1,2,\dots,n\}|}{n}$$ exists, then $\delta(P)$ is called the natural density or asymptotic density of $P$ in $\mathbb N$.
\end{defn}

\begin{note}
	Any subset $P$ of $\mathbb N$ may not have natural density as it depends totally on the existence of the limit.
\end{note}

By using the concept of natural density, in 1951, Fast~\cite{fast}, Steinhaus~\cite{steinhaus} introduced independently the notion of statistical convergence, which is another generalization of usual convergence. 

\begin{defn}A sequence $(p_k)_k$ of real numbers is called convergent statistically to $\ell\in\mathbb R$ if for any $\epsilon>0$, $\delta(\{k\in\mathbb N: |p_k-\ell|>\epsilon\})=0$. We use the notation $p_k\xrightarrow{stat}\ell$ or $stat\lim\limits_{n\to \infty}p_n=\ell$.
\end{defn}

\begin{exmp}
	The sequence $(\lambda_n)_n$ defined by 
	$$\lambda_j=\begin{cases}
	j \quad\text{if } j=k^2,~k\in\mathbb N \\
	0 \quad\text{otherwise }
	\end{cases}$$ is statistically convergent to $0$.
\end{exmp}

Above example shows that unlike usual convergence, a statistically convergent sequence may be unbounded.
Let $st$ be the set of all bounded statistically convergent real sequences.

Throughout this paper, we say that a condition $\mathcal C$ is satisfied by $x_k$ \textit{for almost all} $ k $ (briefly, $\pmb{a.a.k}$) \textbf{if} $E=\{k\in\mathbb N:x_k\text{ does not satisfy condition }\mathcal C$ $\}\implies\delta(E)=0$. We'll now state some well known results which will be used in the sequel.

\begin{lem}\label{lem1}
	Let $C,D\subset\mathbb N$. If $\delta(C)$, $\delta(D)$, $\delta(C\cup D)$ exist, then $\max\{\delta(C),\delta(D)\}\leq\delta(C\cup D)\leq \min \{\delta(C)+\delta(D),1\}$. Furthermore, if $\delta(C)=\delta(D)=0$, then $\delta(C\cup D)$ exists and equals to $0$.
\end{lem}

\begin{lem}\label{lem3}\cite{salat}
	Let $(p_n)_n$ be a real sequence. Then, $p_n\xrightarrow{stat}\ell$ if and only if there exists some $J\subset\mathbb N$ with $J=\{\alpha_1<\alpha_2<\dots <\alpha_n<\dots\}$ such that $\delta(J)=1$ and $\lim\limits_{j\to\infty} p_{\alpha_j}=\ell$.
\end{lem}

\begin{lem}\label{lem4}
	Let $(r_n)_n$  be a statistically convergent real sequence and $stat\lim\limits_{n\to\infty}r_n=\ell$. If $(s_n)_n$ is defined by $s_n=r_n-\ell$ for all $n\in \mathbb N$, then $(s_n)_n$ converges statistically to $0$.
\end{lem}

\begin{lem}\label{lem5}
	Let $(p_n)_n$ be a statistically convergent real sequence. Then, the telescoping sequence of $(p_n)_n$ is statistically convergent to $0$ i.e. $stat\lim\limits_{n\to\infty}(p_n-p_{n+1})=0$.
\end{lem}  
\begin{proof}
Let $stat\lim\limits_{n\to\infty}p_n=l$ and $\epsilon>0$ be any real number. Now, for any $n\in\mathbb N$, $|p_{n+1}-p_n|\leq|p_{n+1}-l|+|p_n-l|$. Thus, $|p_{n+1}-l|\leq \frac{\epsilon}{2}$ and $|p_n-l|\leq \frac{\epsilon}{2}$ $\implies$ $|p_{n+1}-p_n|\leq\epsilon$. Contrapositively, $|p_{n+1}-p_n|>\epsilon$ $\implies$ $|p_{n+1}-l|> \frac{\epsilon}{2}$ or $|p_n-l|> \frac{\epsilon}{2}$. Hence, $\{n\in\mathbb N:|p_{n+1}-p_n|>\epsilon\}\subset\{n\in\mathbb N:|p_{n+1}-l|> \frac{\epsilon}{2}\}\cup\{n\in\mathbb N:|p_n-l|> \frac{\epsilon}{2}\}$. Therefore, $\delta(\{n\in\mathbb N:|p_{n+1}-p_n|>\epsilon\})=0$. Thus $stat\lim\limits_{n\to\infty}(p_n-p_{n+1})=0$.
\end{proof}

\begin{lem}\label{lem6}
	Let $(x_n)_n$ and $(y_n)_n$ be two real sequences such that $x_n\xrightarrow{stat}\ell$ and $y_k=x_k \ a.a.k$. Then, $y_n\xrightarrow{stat}\ell$.
\end{lem}

\begin{proof}
Let $E=\{k\in\mathbb N:y_k\not=x_k \}$ and $\epsilon>0$ be any real number. Then $\delta(E)=0$. Now, $$\{n\in\mathbb N:|y_n-\ell|>\epsilon \}\subset \{n\in\mathbb N:|x_n-\ell|>\epsilon \}\cup E,$$ which implies that $\delta(\{n\in\mathbb N:|y_n-\ell|>\epsilon \})=0$. Hence $y_n\xrightarrow{stat}\ell$.
\end{proof}

\begin{defn}
	A family $\mathcal I$ of subsets of $\mathbb N$ is said to be an ideal of $\mathbb N$ if\\ 
(i) $A\cup B \in \mathcal I$ for each $A,B\in \mathcal I$,\\ 
(ii) $A\subset B$ with $B\in \mathcal I$ implies $A\in \mathcal I$.
\end{defn}
\begin{defn}\cite{kostyyrkosalat2001}
	A real sequence $(z_k)_k$ is said to be $\mathcal I$-convergent to $ l $ if for any $\epsilon>0$, the set $ \{k\in\mathbb N:|z_k-l|>\epsilon \}\in \mathcal I $.
\end{defn}
In \cite{kostyyrkosalat2001} Kostyrko et al.  showed that $\mathcal I_\delta=\{A\subset \mathbb N: \delta(A)=0 \}$ is an ideal and $\mathcal I_\delta$-convergence is actually statistical convergence. The following result is a consequence of Hahn-Banach theorem.

\begin{prop}\label{pro1}
	Let $X$ be a normed linear space and $Y$ be a subspace space of $X$. If $\alpha\in X-\overline{Y}$ and $\mu=d(\alpha,Y):=\inf \{d(\alpha,y):y\in Y\}$, then there exists a bounded linear functional $f:X\to \mathbb R$ such that $f(\alpha)=1$, $f(y)=0, \forall y\in Y$ with $\|f\|=\mu^{-1}$.
\end{prop}

Lorentz~\cite{lorentz} proved that $\mathcal A$ is a closed linear non-separable subspace of $l_\infty$ and $\mathcal A$ is dense in itself but nowhere dense in $\l_\infty$. Also, Lorentz~\cite{lorentz} characterized the almost convergence given as follows:

\begin{prop}\label{prop:lorentz}\cite{lorentz}
	Let $x=(x_n)_n\in l_\infty$. Then $(x_n)_n$ is almost convergent to some $\ell$ if and only if $$\lim\limits_{k\to\infty}\frac{x_p+x_{p+1}+\dots+x_{p+k-1}}{k}=\ell$$
	holds for each $p\in\mathbb N$.
\end{prop} 

\begin{exmp}\label{ex:onlyalmost}\cite{miller}
	The divergent sequencce $(1,0,1,0,\dots)\in l_\infty$ is almost convergent to $\frac{1}{2}$. But it is not statistically convergent.
\end{exmp}

\begin{exmp}\label{ex:onlyst}\cite{miller}
	Consider a sequence $x=(x_n)_n$ in $\{0,1 \}$ constructed as follows:
	$$x=(~\underbrace{0,0,\dots,0}_{100\text{ copies}}~,~   \overbrace{1,1,\dots,1}^{10\text{ copies}}~,~ \underbrace{0,0,\dots,0}_{100^2\text{ copies}}~,~ \overbrace{1,1,\dots,1}^{10^2\text{ copies}}~,~ \underbrace{0,0,\dots,0}_{100^3\text{ copies}}~,~   \overbrace{1,1,\dots,1}^{10^3\text{ copies}}~,~ \dots)$$
	Then, $(x_n)_n$ is statistically convergent to $0$. But it is not almost convergent.
\end{exmp}

In \cite{miller} Miller and Orhan showed that almost convergence and statistical convergence are incomparable, i.e., there exists a statistically convergent (resp. almost convergent) sequence which is not almost convergent (resp. statistically convergent). Thus $st\not\subset\mathcal A$ (see, Example~\ref{ex:onlyst}) and $\mathcal A\not\subset st$ (see, Example~\ref{ex:onlyalmost}). In this paper, we have defined a new type of convergence which is a generalization and comparable with both almost and statistical convergence. To this end we have to introduce the concept of Banach limits in the context of statistical convergence.

\section{Main Results}

Salat~\cite{salat} showed that $st$ is a closed subspace of $l_\infty$. Since, $\phi_n\xrightarrow{stat}a$, $\xi_n\xrightarrow{stat}b$ imply $\phi_n+\xi_n\xrightarrow{stat}a+b$ and $\lambda\xi_n\xrightarrow{stat}\lambda b$ for any $\lambda\in\mathbb R$ and $\phi, \xi\in st$, the map $g:st\to \mathbb R$ defined by $g(x)=stat\lim\limits_{n\to\infty} x_n$ is a linear functional on $st$. We call this functional $g$ by statistical limit functional on $st$.

\begin{lem}
	The linear functional $g:st\to \mathbb R$ defined by $g(x)=stat\lim\limits_{n\to\infty} x_n$ is a bounded with $||g||=1$.
\end{lem}

\begin{proof}
Let $x\in st$. Then $$|g(x)|=|stat\lim\limits_{n\to\infty}x_n |\leq|\sup\limits_{n\in \mathbb N}x_n|\leq\sup\limits_{n\in \mathbb N}|x_n|=||x||_\infty,$$ which implies $\frac{|g(x)|}{||x||_\infty}\leq1$. Thus  $||g||\leq1$.
Again, consider $y=(\lambda,\lambda,\lambda,\dots)\in st$. Then $g(y)=\lambda=||y||_\infty$. Thus $\exists \ y\in st$ such that $\frac{|g(y)|}{||y||_\infty}=1$ $\implies$ $||g||\geq1$. Hence $||g||=1$.
\end{proof}

Thus by Hahn-Banach theorem, $g$ can be extended to $l_\infty$ preserving norm i.e., $\exists$ $L\in (l_\infty)^*$ such that $L|_{st}=g$ and $||L||=||g||$, where $(l_\infty)^*$ is the continuous dual (dual space) of $l_\infty$. We now state and prove the main result of the paper.

\begin{thm}\label{th1}
	There exists a functional $\mathcal F:l_\infty\to\mathbb R$ satisfying the following:
	
	(i) $||\mathcal F||=1$,
	
	(ii)  $\mathcal F|_{st}=g$, where $g$ is the statistical limit functional on $st$,
	
	(iii) If $s\in l_\infty$ with $s_n\geq0$ $a.a. n$, then $\mathcal F(s)\geq0$,
	
	(iv)  If $x\in l_\infty$, then $\mathcal F(x)=\mathcal F(Tx)$, where $T:l_\infty\to l_\infty$ is a map with $(Tx)_k=(Sx)_k \ a.a. k$ for each $x\in l_\infty$ and $S$ is the shift operator defined by $S((x_n)_{n=1}^\infty)=(x_n)_{n=2}^\infty$.
\end{thm}

\begin{proof}
Let $G=\{x-Sx : x\in l_\infty\}$ and $N=\{y\in l_\infty : y_k=x_k \ a.a. k, x\in G\}$. So $G\subset N\subset l_\infty$. Clearly $G$ is a subspace of $l_\infty$. Let $p,q\in N$ and any $\mu, \lambda\in \mathbb R$. Then $\exists$ $x,y\in l_\infty$ such that $p_n=(x-Sx)_n \ a.a.n$ and $q_n=(y-Sy)_n \ a.a.n$. Since, $G$ is a subspace, $\mu(x-Sx)+\lambda(y-Sy)\in G$. So by Lemma \ref{lem1}, we have $(\mu p+\lambda q)_n=(\mu(x-Sx)+\lambda(y-Sy))_n \ a.a.n$ which implies $\mu p+\lambda q\in N$. Therefore $N$ is a subspace of $l_\infty$.

Now, we claim that $d(1,N)=1$, where $1=(1,1,1,\dots)$. For, $0\in G\implies d(1,G)\leq1$. Since, $G\subset N$, $d(1,N)\leq d(1,G)\leq1$. Let $p\in N$. Then $p_n=(b-Sb)_n \ a.a.n$ for some $b\in l_\infty$. If $p_n\leq0$ for some $n\in \mathbb N$, then $||1-p||_\infty=\sup\limits_{n\in \mathbb N}|1-p_n|\geq1$. Again, if $p_n\geq0,\ \forall n\in\mathbb N$, then $(b-Sb)_n\geq0\ a.a. n$ which implies $b_n\geq b_{n+1}\ a.a.n$. Thus $(b_n)_n$ has a subsequence $(b_{n_k})_k$ such that $b_{n_k}\geq b_{n_{k+1}}\ \forall k\in\mathbb N$ with $\delta(\{n_k:k\in\mathbb N \})=1$. Since $(b_{n_k})_k$ is a monotonically decreasing bounded sequence, $\lim\limits_{n\to \infty}b_{n_k}=l$ (say) exists. Thus by Lemma~\ref{lem3} $stat\lim\limits_{n\to\infty}b_n=l$. Therefore using Lemma~\ref{lem5} we have $stat\lim\limits_{n\to\infty}(b_n-b_{n+1})=0$, i.e. $stat\lim\limits_{n\to\infty}(b-Sb)_n=0$. So, by Lemma~\ref{lem6} $stat\lim\limits_{n\to\infty}p_n=0$. Hence using Lemma~\ref{lem3}, there is a subsequence $(p_{n_j})_j$ of $(p_n)_n$ such that $\delta(\{n_j:j\in\mathbb N \})=1$ with $\lim\limits_{j\to\infty}p_{n_j}=0$. So, $||1-p||_\infty=\sup\limits_{n\in \mathbb N}|1-p_n|\geq\sup\limits_{j\in\mathbb N}|1-p_{n_j}|=1$. Thus, for any $p\in N$ we have $||1-p||_\infty\geq1$, which implies $d(1,N)\geq1$. Hence $d(1,N)=1$.

Clearly $1\notin \overline N$. By the Proposition~\ref{pro1} there exists a functional $\mathcal F:l_\infty\to \mathbb R$ such that $\mathcal F(1)=1$, $\mathcal F(y)=0,\ \forall\ y\in N$ with $||\mathcal F||=d(1,N)^{-1}=1$. Let $st_0$ be the collection of all bounded sequences which converge stasistically to $0$. 

We claim that $st_0\subset \ker \mathcal F$. For, let $\xi\in st_0$ and $\epsilon>0$ be any real number. Then $\xi\in l_\infty$ and $\xi_n\xrightarrow{stat}0$ $\implies$ if $U_\epsilon=\{k\in\mathbb N:|\xi_k|>\epsilon \}$, then $\delta(U_\epsilon)=0$. Again, $\mathcal F(y)=0,\ \forall\ y\in N$ implies that $\mathcal F(y)=0,\ y_k=(z-Sz)_k\ a.a.k,\ \forall z\in G$. So by the linearity of $\mathcal F$ we can write $\mathcal F(x)=\mathcal F(T x)$, where $T :l_\infty\to l_\infty$ is any map such that $\sigma_T(z)=\{j\in\mathbb N: (Tz)_j\not=(Sz)_j \}$ and $\delta[\sigma_T(z)]=0,\ \forall z\in l_\infty$.

Choosing $\sigma_T(x)=U_\epsilon$ for all $x\in l_\infty$,  we consider the map $T:l_\infty\to l_\infty$ defined by $Tx=r$, for all $x\in l_\infty$, where
\begin{equation*}
r_k= \begin{cases}
0\quad \text{if }\ k+1\in U_\epsilon\\
x_{k+1}\quad \text{otherwise.}
\end{cases}
\end{equation*}
Then $\mathcal F(\xi)=\mathcal F(T \xi)$ which implies that $|\mathcal F(\xi)|=|\mathcal F(T\xi)|\leq||\mathcal F||\ ||T\xi||_\infty=\sup\{|(T\xi)_k|:k\in\mathbb N \}\leq\epsilon$. Since $\epsilon>0$ is arbitrary, $\mathcal F(\xi)=0$. Thus $st_0\subset \ker \mathcal F$.

Next we claim that $\mathcal F|_{st}=g$. For, let $x\in st$, i.e. $x_n\xrightarrow{stat}\Omega$ (say). Then the sequence $e$ defined by $e_j=x_j-\Omega, \forall j\in \mathbb N$. Now by Lemma~\ref{lem4}, $x-\Omega1=e\in st_0$ which implies $e\in \ker \mathcal F$. Now 
\begin{equation*}
\mathcal F(x)=\mathcal F(x-\Omega1)+\mathcal F(\Omega1)\\
=\mathcal F(e)+\Omega \mathcal F(1)\\
=\Omega \mathcal F(1)\\
=\Omega\\
=stat\lim\limits_{n\to\infty}x_n\\
=g(x).
\end{equation*}
Thus $\mathcal F|_{st}=g$.

Next we show that if $u,v\in l_\infty$ with $u_k=v_k\ a.a.k$, then $\mathcal F(u)=\mathcal F(v)$. For, let $w_n=u_n-v_n,\ \forall n\in \mathbb N$. Let $K=\{k:w_k\not=0 \}$. Then $\delta(K)=0$. Now consider the map $T:l_\infty \to l_\infty$ with $\sigma_T(w)=K$ defined by $Ta=b$, for all $a\in l_\infty$, where
\begin{equation*}
b_k=  \begin{cases}
0\quad\text{if }\  k+1\in K\\
a_{k+1}\quad \text{otherwise.}
\end{cases}
\end{equation*}
Therefore $Tw=0$. Now, $\mathcal F(u)-\mathcal F(v)=\mathcal F(u-v)=\mathcal F(w)=\mathcal F(Tw)=\mathcal F(0)=0$. Thus $\mathcal F(u)=\mathcal F(v)$.

If possible, suppose that there exists $z\in l_\infty$ with $z_n\geq0\ a.a.n$ such that $\mathcal F(z)<0$. Consider $y\in l_\infty$ defined by $y_n=\frac{z_n}{||z||_\infty}$. Clearly $y_n\geq0\ a.a.n$ and $\mathcal F(y)<0$. Again consider $x\in l_\infty$ defined by
\begin{equation*}
x_k=\begin{cases}
0 \quad\text{if } y_k<0\\
y_k \quad\text{if } y_k\geq0.
\end{cases}
\end{equation*}
Then, $x_n=y_n,\ a.a.n$ and $0\leq x_n\leq 1,\ \forall n\in \mathbb N$. Therefore $\mathcal F(x)=\mathcal F(y)<0$ and $$||1-x||_\infty=\sup\limits_{n\in \mathbb N}|1-x_n|\leq1.$$ Again $\mathcal F(1-x)=\mathcal F(1)-\mathcal F(x)=1-\mathcal F(x)>1$. Thus we get $1<|\mathcal F(1-x)|\leq||\mathcal F||\ ||1-x||_\infty=||1-x||_\infty\leq1$, which is a contradiction. Hence, if $s\in l_\infty$ with $s_k\geq0\ a.a.k$, then $\mathcal F(s)\geq0$. This completes the proof.
\end{proof}

The functional $\mathcal F$ in the Theorem~\ref{th1} is named as \textit{Banach statistical limit} functional and defined as follows:

\begin{defn}\label{def:bsl}
	A functional $\mathcal F:l_\infty\to\mathbb R$ is called a Banach statistical limit if it satisfies the following : 
	
	(i) $||\mathcal F||=1$,
	
	(ii)  $\mathcal F|_{st}=g$, where $g$ is the statistical limit functional on $st$,
	
	(iii) If $s\in l_\infty$ with $s_n\geq0$ $a.a. n$, then $\mathcal F(s)\geq0$,
	
	(iv)  If $x\in l_\infty$, then $\mathcal F(x)=\mathcal F(Tx)$,	where $T:l_\infty\to l_\infty$ is a map with $(Tx)_k=(Sx)_k \ a.a. k$ for each $x\in l_\infty$ and $S$ is the shift operator defined by $S((x_n)_{n=1}^\infty)=(x_n)_{n=2}^\infty$.
\end{defn}

\begin{cor}\label{cor:bslf}
	Every Banach statistical limit functional is a Banach limit functional on $l_\infty$.
\end{cor}

\begin{cor}
	Let $\mathcal F$ be any Banach statistical limit functional on $l_\infty$. If $u,v\in l_\infty$ with $u_k=v_k\ a.a.k$, then $\mathcal F(u)=\mathcal F(v)$.
\end{cor}

Now, by using Banach statistical limit functional, we introduce a new type of convergence, called, \textit{generalized almost statistical} convergence (briefly, GAS convergence), defined as follows:

\begin{defn}\label{def:asc}
	Let $x\in l_\infty$. Then $x$ is said to be \textbf{generalized almost statistically convergent} to $\lambda$, if $\mathcal F(x)=\lambda$ for all Banach statistical limit functionals $\mathcal F$, i.e., if $\mathcal F(x)$ is invariant (unique) for each Banach statistical limit fnctionals $\mathcal F$ on $l_\infty$.
\end{defn}
 Let $\mathcal S$ be the set of all GAS convergent real sequences. The following result is easily obtained from the Definitions~\ref{def:bsl} and \ref{def:asc}.

\begin{cor}\label{cor:stsubsets}
	Every statistically convergent sequence is GAS convergent with the same limit, i.e. $st\subset \mathcal S$.
\end{cor}
The converse is not true (see Example~\ref{ex:onlyalmoststat1}). Clearly $st\subsetneq \mathcal S$.

From the Definitions~\ref{def:bsl} and \ref{def:asc}, it follows that every bounded statistically convergent sequence is GAS convergent with the same limit.

\begin{lem}
	Every almost convergent real sequence is GAS convergent with the same limit, i.e. $\mathcal A\subset\mathcal S$.
\end{lem}
\begin{proof}
Suppose $r=(r_i)_i\in l_\infty$ is an almost convergent real sequence with limit $\kappa$. Then for any Banach limit functional $B:l_\infty \to \mathbb R$ we have $B(r)=\kappa$. By Corollary~\ref{cor:bslf}, we can easily say that $\mathcal F(r)=\kappa$ for any Banach statistical limit functional $\mathcal F$. Thus $(r_n)_n$ is generalized almost statistically convergent to $\kappa$.
\end{proof}

The converse is not true. Because in Example~\ref{ex:onlyst} the sequence $x=(x_n)_n$ is not almost convergent. But it is statistically convergent, which implies $x\in \mathcal S$ (by Corollary~\ref{cor:stsubsets}). Clearly $\mathcal A\subsetneq\mathcal S$.

\begin{exmp}\label{ex:onlyalmoststat1}
	Consider the real bounded sequence $x=(x_n)_n$ defined by
	$$
	x_k = 
	\begin{cases}
	5 \quad\text{if } k\text{ is a perfect square number}    \\
	0 \quad\text{if } k\text{ is even and not a perfect square} \\
	1 \quad\text{if } k\text{ is odd and not a perfect square}
	\end{cases}
	$$
	i.e., $(x_n)_n=(5,0,1,5,1,0,1,0,5,0,1,0,1,0,1,5,1,0,1,0,1,0,1,0,5,0,\dots)$\newline
	Clearly, $x$ is not convergent in usual sense. Even $x$ is not convergent statistically. But $x$ is GAS convergent to $\frac{1}{2}$. For, let $S=\{n^2:n\in \mathbb N \}$. Now consider the map $T:l_\infty \to l_\infty$ (as in the Theorem~\ref{th1}) with $\sigma_T(x)=S$ defined by $Tz=y$, for all $z\in l_\infty$, where
	$$
	y_k = 
	\begin{cases}
	z_{k+1} \quad\text{if } k+1\notin S    \\
	0 \quad\text{if } k+1\in S  \text{ and } k \text{ is odd} \\
	1 \quad\text{if } k+1\in S  \text{ and } k \text{ is even.}
	\end{cases}
	$$
	Then $Tx=(0,1,0,1,0,1,\dots)$. Let $\mathcal F$ be any Banach statistical limit functional. Since, $u,v\in l_\infty$ with $u_k=v_k\ a.a.k$ implies $\mathcal F(u)=\mathcal F(v)$ (as shown in the proof of the Theorem~\ref{th1}), then $\mathcal F(x)=\mathcal F(1,0,1,0,1,0,\dots)$.  Then by the Theorem~\ref{th1} we have $\mathcal F(x)=\mathcal F(Tx)=\mathcal F(0,1,0,1,0,1,\dots)=\mathcal F((1,1,1,1,\dots)-(1,0,1,0,\dots))=\mathcal F(1,1,1,1,\dots)-\mathcal F(1,0,1,0,\dots)=1-\mathcal F(x)$. Thus $\mathcal F(x)=\frac{1}{2}$.
\end{exmp}

The following example shows that there exists a GAS convergent sequence, which is neither almost convergent nor statistically convergent. i.e. $st\cup \mathcal A\subsetneq \mathcal S$.
\begin{exmp}\label{ex:onlygalst}
	Consider the sequence $\xi=(\xi_n)_n$ defined as follows:
	$$\xi=(~\underbrace{1,0,1,0,\dots}_{100\text{ terms}}~,~   \overbrace{1,1,\dots,1}^{10\text{ terms}}~,~ \underbrace{1,0,1,0,\dots}_{100^2\text{ terms}}~,~ \overbrace{1,1,\dots,1}^{10^2\text{ terms}}~,~ \underbrace{1,0,1,0,\dots}_{100^3\text{ terms}}~,~   \overbrace{1,1,\dots,1}^{10^3\text{ terms}}~,~ \dots).$$
	It is easy to check that	 $\xi$ is neither a statistically convergent nor an almost convergent sequence. But $\xi$ is GAS convergent. For, let $E=\bigcup\limits_{i=1}^\infty(b_i-10^i,b_i]\cap 2\mathbb N$ with $b_i=\sum\limits_{j=1}^i (100^j+10^j)$ for each $i\in \mathbb N$. Let $T:l_\infty\to l_\infty$ defined by $Tx=z$ with 
	$$z_k=\begin{cases}
	x_{k+1}, \quad \text{if $k+1\notin E$}\\
	0,\quad \text{otherwise.}
	\end{cases}$$
	Then $z=(0,1,0,1,0,\dots)$. Since $\delta(E)=0$, $z_k=(Sx)_k\ a.a.k$.	Therefore, $\mathcal F(x)=\mathcal F(Tx)=\mathcal F(0,1,0,1,0,\dots)=\mathcal F(1,0,1,0,\dots)=\frac{1}{2}$.
\end{exmp}

\begin{thm}
GAS convergence can not be characterized by ideal convergence for proper ideals of $\mathbb N$.
\end{thm}
\begin{proof}
	If possible, suppose that GAS convergence coincides with ideal convergence for some proper ideal $\mathcal I$ of $\mathbb N$. Since, $\xi=(1,0,1,0,\dots)$ is almost convergent to  $\frac{1}{2}$, $\xi$ is GAS convergent to $\frac{1}{2}$. Hence, for any $\epsilon>0$, $A_\epsilon:=\{k\in\mathbb N:|\xi_k-\frac{1}{2}|\geq \epsilon \}\in\mathcal I$. But $A_\frac{1}{2}=\mathbb N\in \mathcal I$, which contradicts that $\mathcal I$ is a proper ideal of $\mathbb N$.
\end{proof}

\begin{prop}\label{prop:topology}The following are some topological properties of $\mathcal S$\\
	(i) $\mathcal S$ is closed\\
	(ii) $\mathcal S$ is non-separable in $l_\infty$\\
	(iii) $\mathcal S$ is first countable but not second countable.\\
	(iv) $\mathcal S$ is not Lindel\"of and not compact.
\end{prop}

\begin{proof}
	(i) $\mathcal S$ is closed. For, let $s\in \bar{\mathcal  S}$. Then, by the sequence lemma, there exists a sequence $(s^{(n)})_{n\in\mathbb N}$ in $\mathcal S$ such that $\lim\limits_{n\to \infty}s^{(n)}=s$. Let $\rho, \tau$ be any two Banach statistical limit functional. Since $\rho, \tau$ are continuous, 
	$\rho(s)
	=\rho(\lim\limits_{n\to\infty}s^{(n)})
	=\lim\limits_{n\to\infty}(\rho (s^{(n)}))
	=\lim\limits_{n\to\infty}(\tau (s^{(n)}))
	=\tau(\lim\limits_{n\to\infty}s^{(n)})
	=\tau(s)$. So $s\in {\mathcal S}$. Hence $\mathcal S$ is closed.
	
	(ii) $\mathcal S$ is non-separable. For, consider $\Lambda\subset l_\infty$ defined as follows: $$\Lambda=\left\lbrace x\in l_\infty :x_k =\begin{smallmatrix}
\begin{cases}
0 \text{ or }1 \quad\text{if } k \text{ is a perfect square}    \\
0 \quad\text{otherwise. }
\end{cases} 
	\end{smallmatrix} \right\rbrace$$
	Clearly $\Lambda$ is an uncountable subset of $\mathcal A\subset\mathcal S$, which implies that $\mathcal S$ is uncountable. Now for any distinct $u,v\in\Lambda$, $d(u,v)=||u-v||_\infty=1$. Thus $\Lambda$ is an uncountable discrete subset of $\mathcal S$. Let $D$ be any dense set in $\mathcal S$ i.e. $\bar D=\mathcal S$. Let us consider any $s, t\in\Lambda$ with $s\not=t$. Then $s,t\in\mathcal S=\bar D$. Then the disjoint open balls $B_d(s,\frac{1}{2})$ and $B_d(t,\frac{1}{2})$ must have non-empty intersections with $D$, which implies that there are two distinct elements of $D$. Since $\Lambda$ is uncountable, $D$ is also uncountable. Thus $\mathcal S$ does not contain any countable dense subset, i.e. $\mathcal S$ is non-separable. 
	
	(iii) $\mathcal S$ is not second countable due to its non-separability.
	
	(iv) Since $\mathcal S$ is metrizable and not second countable, $\mathcal S$ is non-Lindel\"of and non-compact.
\end{proof}

\section{Conclusion}
The notion of GAS convergence is introduced in this paper as a generalization of almost convergence as well as statistical convergence of bounded real sequences. The existence of GAS convergence is ensured by the existence of Banach statistical limit functional (see Theorem \ref{th1}). Some topological properties of the space of all GAS convergent sequences are obtained (see Proposition \ref{prop:topology}). The GAS convergence can not be characterized by ideal convergence for proper ideals of $\mathbb N$. Open problem: In the analogy of Proposition \ref{prop:lorentz}, is there any necessary and sufficient condition for GAS convergence?

\newpage
\textbf{Acknowledgement:}  The second author is grateful to The Council Of Scientific and Industrial Research (CSIR), Government of India, for the award of JRF (Junior Research Fellowship).

\end{document}